\documentclass[11pt]{amsart}
\usepackage{amsmath,amsthm,amsfonts,amssymb,latexsym,epsfig}

\newtheorem{theorem}{Theorem}[section]

\newtheorem{lemma}{Lemma}[section]

\numberwithin{equation}{section}

\theoremstyle{definition}

\theoremstyle{remark}

\begin{document}
\title{The $l^p$ norms of a class of weighted mean matrices}
\author{Peng Gao and Huayu Zhao}

\subjclass[2000]{Primary 47A30} \keywords{Hardy's inequality, $l^p$
operator norms, weighted mean matrices}


\begin{abstract}
  We study the $l^p$ norms of a class of weighted mean matrices whose diagonal terms are given by $n^{\alpha}/\sum^{n}_{i=1}i^{\alpha}$ with $\alpha >
  -1$. The $l^p$ norms of such matrices are known for $p \geq 2, (\alpha+1)p >1$ and $1<p \leq 4/3, 1/p \leq \alpha \leq 1$.  In this paper, we determine
  the $l^p$ norms of such matrices for $p \geq 1.35, 0\leq \alpha \leq 1$.
\end{abstract}

\maketitle
\section{Introduction}
\label{sec 1} \setcounter{equation}{0}

  For $p>1$, we denote $l^p$ for the Banach space of all complex sequences ${\bf a}=(a_n)_{n \geq 1}$ with norm
\begin{equation*}
   ||{\bf a}||_p: =(\sum_{n=1}^{\infty}|a_n|^p)^{1/p} < \infty.
\end{equation*}

  Let $C=(c_{n,k})$ be an infinite matrix. The $l^{p}$ operator norm of $C$ is defined as
\begin{equation*}
\label{02}
    ||C||_{p,p}=\sup_{||{\bf a}||_p = 1}\Big | \Big |C \cdot {\bf a}\Big | \Big |_p.
\end{equation*}

   A matrix $A=(a_{n,k})$ is said to be a lower triangular matrix if $a_{n,k}=0$ for $n<k$ and a lower triangular matrix $A$ is said to be a summability
   matrix if $a_{n,k} \geq 0$ and $\sum^n_{k=1}a_{n,k}=1$. We further say that a summability matrix $A$ is a weighted mean matrix if its entries satisfy:
\begin{equation}
\label{021}
    a_{n,k}=\frac {\lambda_k}{\Lambda_n},  ~~ 1 \leq k \leq
    n; \hspace{0.1in} \Lambda_n=\sum^n_{i=1}\lambda_i, \lambda_i \geq 0, \lambda_1>0.
\end{equation}
  We shall also say that a weighted mean matrix $A$ is generated by $\{\lambda_n \}^{\infty}_{n=1}$ when $A$ is an infinite weighted mean matrix whose
  entries are given by \eqref{021}.

  The study on $l^{p}$ norms of weighted mean matrices originated from the celebrated Hardy's inequality (\cite[Theorem 326]{HLP}), which asserts that for
  $p>1$,
\begin{equation*}
\sum^{\infty}_{n=1}\Big{|}\frac {1}{n}
\sum^n_{k=1}a_k\Big{|}^p \leq \Big (\frac
{p}{p-1} \Big )^p\sum^\infty_{n=1}|a_n|^p,
\end{equation*}

  In terms of $l^p$ norms, Hardy's inequality can be regarded as asserting that the Ces\'aro matrix operator $C$, given by $c_{n,k}=1/n , k\leq n$ and $0$
otherwise, has norm $p/(p-1)$.

  In a series of work (see for example \cite{B1}-\cite{Be1}), G. Bennett studied systematically $l^p$ norms of general weighted mean matrices. A particular
  class of matrices considered by Bennett is the weighted mean matrices generated by $\{ n^{\alpha} \}^{\infty}_{n=1}$. More concretely, in \cite[p.
  40-41]{B4} (see also \cite[p. 407]{B5}), Bennett claimed the following inequality to hold for any ${\bf a} \in l^p$ whenever $p>1, (\alpha+1) p >1$:
\begin{align}
\label{8}
   \sum^{\infty}_{n=1}\Big{|}\frac
1{\sum^n_{i=1}i^{\alpha}}\sum^n_{i=1}i^{\alpha}a_i\Big{|}^p
\leq \Big(\frac {(\alpha+1) p}{(\alpha+1) p-1} \Big
)^p\sum^{\infty}_{n=1}|a_n|^p.
\end{align}
     We note here the constant $((\alpha+1) p /((\alpha+1) p-1))^p$ is best possible (see \cite{Be1}).

No proof of inequality \eqref{8} is given in \cite{B4}-\cite{B5}. Bennett \cite{Be1} and the first-named author \cite{G}
     proved inequality \eqref{8} for $p>1, \alpha \geq 1$ or $\alpha \leq 0, (\alpha+1) p >1$
     independently. The proofs of \eqref{8} given in \cite{Be1} and \cite{G} are the same, they both use the following well-known result of J. M. Cartlidge
     \cite{Car} (see also \cite[p. 416, Theorem C]{B1}), which says that for any weighted mean matrix $A$ given by \eqref{021} and any fixed $p>1$, if
\begin{equation*}
    L=\sup_n\Big(\frac {\Lambda_{n+1}}{\lambda_{n+1}}-\frac
    {\Lambda_n}{\lambda_n}\Big) < p ~~,
\end{equation*}
    then $||A||_{p,p} \leq p/(p-L)$.

    An improvement of the above result of Cartlidge is obtained by the first-named author in \cite{G5}, who showed that
if for any integer $n \geq 1$, there exists a positive constant
    $0<L<p$ such that
\begin{equation}
\label{024}
    \frac {\Lambda_{n+1}}{\lambda_{n+1}} \leq \frac
    {\Lambda_n}{\lambda_n}  \Big (1- \frac
    {L\lambda_n}{p\Lambda_n} \Big )^{1-p}+\frac {L}{p}~~,
\end{equation}
    then
    $||A||_{p,p} \leq p/(p-L)$.

   As an application, the first-named author proved in \cite{G5} that inequality \eqref{8} holds for $p \geq 2, 0 \leq
   \alpha \leq 1$. Using an similar approach earlier, the case  or $1 < p \leq 4/3, 1/p \leq \alpha \leq 1$ is also proved in \cite{G6}.

   In view of the above results, we see that the only open cases for inequality \eqref{8} are $1<p \leq 4/3$, $0<\alpha< 1/p$ and
$4/3< p <2$, $0<\alpha< 1$. It is our goal in this paper to refine the arguments in \cite{G5} to study inequality \eqref{8} for these remaining cases. Our
main
result is
\begin{theorem}
\label{mainthm}
     Inequality \eqref{8} holds for $p \geq 1.35$ and $0 \leq \alpha \leq 1$ or $p >1$ and $1/p \leq
   \alpha \leq 1$.
\end{theorem}

    Our proof of Theorem \ref{mainthm} involves a careful analysis of \eqref{024} as well as the approach used in \cite{G6}. Our arguments in fact show
    that inequality \eqref{8} holds for $p \geq p_0$ and $0 \leq \alpha \leq 1$ for some $p_0$ being slightly less than $1.35$, but we will be content with
    the statement given in Theorem \ref{mainthm} here.
\section{Some Lemmas}
\label{sec 3} \setcounter{equation}{0}
\begin{lemma}
 \label{lem10} Let $\alpha_0$ denote the unique root of $$v(\alpha):= 2(\alpha+1)-(1+\alpha)^2\ln 2-2^{\alpha}=0$$ on the interval $(0,1)$ and let $p_1$
 satisfies
\begin{equation*}
   \frac 1{2^{\alpha_0}} = \frac 1{\alpha_0+1}+\frac {1-1/p_1}{2(\alpha_0+1)^2}.
 \end{equation*}
  Then for $p \geq p_1$ and $0 \leq \alpha \leq 1$, we have
 \begin{align}
\label{nequal1}
    \frac 1{2^{\alpha}} \leq \frac 1{\alpha+1}+\frac {1-1/p}{2(\alpha+1)^2}.
 \end{align}
 \end{lemma}
 \begin{proof}
    We write \eqref{nequal1} as
 \begin{equation*}
    u(\alpha):= \frac {(\alpha+1)^2}{2^{\alpha}}-\alpha \leq \frac {3-\frac 1p}2.
 \end{equation*}
    Note that
 \begin{equation*}
    u'(\alpha)=\frac {v(\alpha)}{2^{\alpha}}.
 \end{equation*}
    Note that $v(0)>0, v(1)<0$ and
\begin{equation*}
   v''(\alpha)=-(2+2^{\alpha}\ln 2)\ln 2<0.
\end{equation*}
   It follows that there exists a unique root $\alpha_0$ of $v(\alpha)=0$ on the interval $(0,1)$ and $u(\alpha)$ reaches its absolute maximal value
on $[0,1]$ at $\alpha_0$. The assertion of the lemma follows from this easily.
 \end{proof}

\begin{lemma}
\label{lem1}
  Let $4/3 \leq  p \leq 3/2$ and $L=1/(1+\alpha)$ with $0 \leq \alpha \leq 1/p$. We have for $\frac 1n \leq x \leq 1$,
\begin{align*}
   (1-\frac Lp x)^{1-p} \geq 1+\frac {(p-1)L}{p} x+ \frac 12(1-\frac 1p)\frac {3n}{3n-1} L^2 x^2.
\end{align*}
\end{lemma}
\begin{proof}
   We see by Taylor expansion that
\begin{align*}
   (1-\frac Lp x)^{1-p} \geq & 1+\frac {(p-1)L}{p} x+ \frac 12(1-\frac 1p)L^2 x^2 \\
&+\frac 16(1-\frac 1p)\frac {1+p}{p}L^3 x^3\sum^{\infty}_{n=0}\left (\prod^{n-1}_{i=0}\frac {p+2+i}{4+i} \right ) \left ( \frac {Lx}{p} \right )^n ,
\end{align*}
  where we define the empty product to be $1$.

  Using the estimation that for all $i \geq 0$, $p \geq 4/3$,
\begin{align*}
    \frac {6/5(p+2+i)}{4+i} \geq 1,
\end{align*}
  we see that
\begin{align*}
   (1-\frac Lp x)^{1-p} \geq & 1+\frac {(p-1)L}{p} x+ \frac 12(1-\frac 1p)L^2 x^2+\frac 16(1-\frac 1p)\frac {1+p}{p}L^3 x^3\sum^{\infty}_{n=0}\left ( \frac
   {5Lx}{6p} \right )^n \\
\geq &  1+\frac {(p-1)L}{p} x+ \frac 12(1-\frac 1p)L^2 x^2+\frac 16(1-\frac 1p)\frac {1+p}{p}L^3 x^3\frac 1{1-5Lx/(6p)} \\
\geq &  1+\frac {(p-1)L}{p} x+ \frac 12(1-\frac 1p)L^2 x^2+\frac 16 \cdot \frac {3n}{3n-1}(1-\frac 1p)\frac {1+p}{p}L^3 x^3,
\end{align*}
   since we have
\begin{align*}
    \frac {5Lx}{6p} \geq \frac {5x}{6(1+p)} \geq \frac {5/n}{6(1+3/2)}=\frac 1{3n}.
\end{align*}

   By further noting that
\begin{align*}
   \frac 16 \cdot  \frac {3n}{3n-1}(1-\frac 1p)\frac {1+p}{p}L^3 x^3 \geq \frac 12 \cdot \frac {1}{3n-1}(1-\frac 1p)L^2 x^2,
\end{align*}
   the desired result follows.
\end{proof}

\begin{lemma}
\label{lem2}
  Let $0 \leq x \leq \frac 1n$ and $0 \leq \alpha \leq 1$. We have
\begin{align*}
   (1+x)^{-\alpha} \geq 1-\alpha x+ \alpha (\alpha+1)d_nx^2, \quad d_n :=\frac {n}{2(n+1)}.
\end{align*}
\end{lemma}
\begin{proof}
   We let
\begin{align*}
   s_{\alpha}(x)=(1+x)^{-\alpha}-1+\alpha x-\alpha (\alpha+1)d_nx^2.
\end{align*}
   Then
\begin{align*}
   s'_{\alpha}(x)= & -\alpha (1+x)^{-\alpha-1}+\alpha -2\alpha (\alpha+1)d_nx, \\
   s''_{\alpha}(x)= & \alpha (\alpha+1)\left ( (1+x)^{-\alpha-2}-2d_n \right ).
\end{align*}

   We note that $s''_{\alpha}(x)$ is a decreasing function of $x$ with $s''_{\alpha}(0)>0, s''_{\alpha}(1/n)<0$. It follows that $s'_{\alpha}(x)$ is first
   increasing and then decreasing
on $[0, 1/n]$. As $s'_{\alpha}(0)=0$, this implies that $s_{\alpha}(x)$ is either increasing on $[0, 1/n]$ or first increasing and then decreasing on
$[0,1/n]$. In either case,
we have that $s_{\alpha}(x) \geq \min \{ s_{\alpha}(0), s_{\alpha}(1/n) \}$ on $[0,1/n]$. Since $s_{\alpha}(0)=0$, it suffices to show $s_{\alpha}(1/n)
\geq 0$. Now, we write
\begin{align*}
   t(\alpha):=s_{\alpha}( \frac 1n)=(1+\frac 1n)^{-\alpha}-1+ \frac {\alpha}n-\frac {\alpha (\alpha+1)d_n}{n^2}.
\end{align*}
   Then we have
\begin{align*}
   t''(\alpha)=\ln^2(1+\frac 1n) \cdot (1+\frac 1n)^{-\alpha}-\frac {2d_n}{n^2}.
\end{align*}

   Note that $t''(\alpha)$ is a decreasing function of $\alpha$ such that $t''(0)=\ln^2(1+\frac 1n)-\frac {2d_n}{n^2} \leq 0$. Thus
$t(\alpha)$ is concave down on $[0, 1]$ so that $t(\alpha) \geq \min \{ t(0), t(1) \}=0$ on $[0,1]$. The assertion of the lemma now follows from this.
\end{proof}

\begin{lemma}
\label{1stlemma}
   Inequality \eqref{8} holds for $1 < p \leq \frac 1{2(1-\ln 2)}, 1/p \leq
   \alpha \leq 1$.
\end{lemma}
\begin{proof}
   It is shown in \cite[p. 51]{G6} that in order for inequality \eqref{8} to be valid, it suffices to show that $w_{n,p}(x) \geq 0$ for $0 \leq x \leq 1$,
   where
\begin{equation*}
  w_{n,p}(x)=x \ln (1+1/n)-\frac 1{p}\ln(1+ \frac {x+1-1/p}{n})-(1-\frac 1{p})\ln(1+ \frac
  {x-1/p}{n}), \quad n \geq 1, p>1.
\end{equation*}
  It is also shown in \cite[p. 51]{G6} that in order for $w_{n,p}(x) \geq 0$, it suffices to show that
\begin{equation*}
   w'_{n,p}(1/p)=\ln (1+1/n)-\frac 1{n}+\frac 1{pn(n+1)} \geq 0.
\end{equation*}
  We now define
\begin{equation*}
  r_p(y)=\ln (1+y)-y+\frac {y^2}{p(1+y)}, \quad 0<y \leq \frac 12.
\end{equation*}
   Note that $w'_{n,p}(1/p)=r_p(1/n)$ and that we have
\begin{equation*}
   r'_p(y)=\frac 1{1+y}-1+\frac {2y+y^2}{p(1+y)^2}.
\end{equation*}
   We further define
\begin{equation*}
   a_p(z)=z-1+\frac 1p(1-z^2), \quad \frac 23 \leq z <1.
\end{equation*}
   It is easy to see that $a(z) \geq 0$ for any $\frac 23 \leq z <1$ when $p<5/3$. By setting $z=1/(1+y)$ implies that $r_p(y)$ is increasing on $0<y \leq
   \frac 12$. As $r_p(0)=0$, it follows that $w'_{n,p}(1/p)>0$ for $n \geq 2, 1<p<5/3$.

    When $n=1$, we have $w'_{1,p}(1/p)=\ln 2-1+\frac 1{2p}$ and  $w'_{1,p}(1/p) \geq 0$ when $p \leq \frac 1{2(1-\ln 2)} <5/3$ and this completes the
    proof.
\end{proof}
\section{Proof of Theorem \ref{mainthm}}
\label{sec 4} \setcounter{equation}{0}

 We assume that $0 \leq \alpha \leq 1$ is being fixed throughout the proof and we write $L=(1+\alpha)^{-1}$. In order for inequality \eqref{8} to valid for
 this $\alpha$, it suffices to show that inequality \eqref{024} is valid for all $n$. Note that when
$n=1$, inequality \eqref{024} is equivalent to \eqref{nequal1} and calculation shows that $p_1 \approx 1.223$ so that inequality \eqref{024} is valid when
$n=1$ for all $0 \leq \alpha \leq 1$ and $p \geq 1.23$ by Lemma \ref{lem10}.

  We assume $n \geq 2$ from now on. Suppose that for some constant $c_n \geq 1$, we have
\begin{align}
\label{cond1}
   (1-\frac Lp \cdot \frac {\lambda_n}{\Lambda_n})^{1-p} \geq 1+\frac {(p-1)L}{p} \cdot  \frac {\lambda_n}{\Lambda_n}+ \frac 12(1-\frac 1p)c_nL^2 \left (
   \frac {\lambda_n}{\Lambda_n} \right )^2.
\end{align}
   Then it is easy to see that inequality \eqref{024} follows from $f_{n,\alpha}(x_n) \geq 0$, where $x_n=(\sum^{n}_{k=1}k^{\alpha})^{-1}$ and
\begin{equation*}
  f_{n,\alpha}(x)=1+n^{\alpha}x\Big (\frac 1{\alpha+1}+\frac {n^{\alpha}}{2}\Big(1-\frac 1{p}\Big) \frac {c_n x}{(\alpha+1)^2} \Big )-\frac
  {n^{\alpha}}{(n+1)^{\alpha}}\Big ( 1+(n+1)^{\alpha}x \Big ).
\end{equation*}

    Note that for fixed $n$, $f_{n,\alpha}(x)$ is a concave up quadratic function of $x$ with $f_{n,\alpha}(0) \geq 0$ and the only root of
    $f'_{n,\alpha}(x)=0$ is $\alpha(\alpha+1)/(c_n n^{\alpha}(1-1/p))$. Moreover, we note that it follows from \cite[Lemma 6.2]{G5} that
\begin{align*}
   x_n \leq \frac {\alpha+1}{n(n+1)^{\alpha}}.
\end{align*}

  Suppose that we have
\begin{equation}
\label{7.2}
  \frac {\alpha+1}{n(n+1)^{\alpha}} \geq \frac {\alpha(\alpha+1)}{c_n n^{\alpha}(1-1/p  )}.
\end{equation}
  Then it suffices to show that for fixed $0 \leq \alpha \leq 1$ and any $n$,
\begin{equation*}
  f_{n,\alpha} \Big (\frac {\alpha(\alpha+1)}{c_nn^{\alpha}(1-1/p)} \Big )= 1-\frac {\alpha^2}{2c_n(1-1/p)}-\frac {n^{\alpha}}{(n+1)^{\alpha}} \geq  0.
\end{equation*}
  The above inequality is already proven to be valid on \cite[p. 846-847]{G5}, since $c_n \geq 1$ here. We are thus led to the consideration to the other
  case, when inequality \eqref{7.2} reverses and we then deduce that when $n \geq 2$ and $0 \leq \alpha \leq 1$,
\begin{equation*}
  f_{n,\alpha}(x_n) \geq f_{n,\alpha} \Big (\frac {\alpha+1}{n(n+1)^{\alpha}} \Big )=\frac {n^{\alpha}}{(n+1)^{\alpha}}g_{n,\alpha}(\frac 1{n}),
\end{equation*}
   where
\begin{equation*}
   g_{n,\alpha}(y)=(1+y)^{\alpha}+y\Big (1+ \frac {1-1/p}{2}c_n y(1+y)^{-\alpha} \Big )-1-(\alpha+1)y.
\end{equation*}

   We now let $k$ be a positive integer and note that by Taylor expansion and Lemma \ref{lem2}, we have for $0 \leq \alpha  \leq 1$, $0<y \leq 1/k$,
\begin{equation*}
   (1+y)^{\alpha} \geq 1+\alpha y+\frac {\alpha (\alpha-1)}{2}y^2, \hspace{0.1in} (1+y)^{-\alpha} \geq  1-\alpha y+ \frac {\alpha (\alpha+1)k}{2(k+1)}y^2.
\end{equation*}

   We also note that the function
\begin{equation*}
   y \mapsto  1-\alpha y+ \frac {\alpha (\alpha+1)k}{2(k+1)}y^2, \quad 0 \leq y \leq \frac 1k
\end{equation*}
    is minimized at $y=\frac 1k$.

   We conclude from the above discussions that when $0 \leq y \leq 1/k$, $0 \leq \alpha \leq 1$,
\begin{align*}
   g_{n,\alpha}(y) \geq & 1+\alpha y+\frac {\alpha (\alpha-1)}{2}y^2+y\Big (1+ \frac {1-1/p}{2}c_ny(1-\alpha y+\frac {\alpha (\alpha+1)k}{2(k+1)}y^2) \Big
   )-1-(\alpha+1)y \\
   =& \frac {y^2}{2}\Big (\alpha (\alpha-1)+(1-1/p)c_n(1-\alpha y+\frac {\alpha (\alpha+1)k}{2(k+1)}y^2) \Big ) \\
\geq & \frac {y^2}{2}\Big (\alpha (\alpha-1)+(1-1/p)c_n\frac {\alpha^2-(2k+1)\alpha+2k(k+1)}{2k(k+1)} \Big ) .
\end{align*}

    It follows that in order for $g_{n,\alpha}(y) \geq 0$, it suffices to have
\begin{align}
\label{pbound}
   p\geq (1-\frac {h_k(\alpha)}{c_n})^{-1},
\end{align}
   where
\begin{eqnarray*}
   h_k(\alpha)=\frac {2k(k+1)\alpha(1-\alpha)}{\alpha^2-(2k+1)\alpha+2k(k+1)}.
\end{eqnarray*}

   We note that
\begin{eqnarray*}
   h'_k(\alpha)=\frac {4k^2(k+1)(\alpha^2-2(k+1)\alpha+k+1)}{(\alpha^2-(2k+1)\alpha+2k(k+1))^2}.
\end{eqnarray*}

   We then deduce that $h'_k(\alpha)$ has a unique root $\alpha=k+1-\sqrt{k(k+1)}$ in $(0,1)$. As $h_k(0)=h_k(1)=0$, it follows that $h_k(\alpha)$ takes
   its maximal value at $\alpha=k+1-\sqrt{k(k+1)}$. In which case we have
\begin{align}
\label{hmin}
   h_k(k+1-\sqrt{k(k+1)})=\frac {2k(k+1)(\sqrt{k(k+1)}-k)}{2k^2+k+(2k-1)\sqrt{k(k+1)}}.
\end{align}

   We now note that inequality \eqref{cond1} is valid for $c_n=1$ by Taylor expansion. Setting $c_n=1$ in \eqref{pbound} and applying \eqref{hmin} by
   taking $k=2$, we see that inequality \eqref{024} is valid for $n \geq 2$ and $0 \leq \alpha \leq 1$ when
\begin{align*}
   p \geq (1- h_2(3-\sqrt{6}))^{-1} \approx 1.451.
\end{align*}

   Combining this bound of $p$ and the bound that $p \geq 1.23$ by Lemma \ref{lem10} for the case $n=1$, we see that inequality \eqref{8} is valid
for all $p \geq 1.46$ and $0 \leq \alpha \leq 1$. We then deduce from this and Lemma \ref{1stlemma} by noting that $1/(2(1-\ln2)) \approx 1.62$ that
inequality \eqref{8} holds for $p >1$ and $1/p \leq \alpha \leq 1$. This completes the proof for the second assertion of Theorem \ref{mainthm}.

   For the first assertion of Theorem \ref{mainthm}, we note by our discussions above that it suffices to further assume that $4/3 \leq p \leq 3/2$ and $0
   \leq \alpha \leq 1/p$.  In which case both Lemma \ref{lem1} and Lemma \ref{lem2} are valid for $x=1/n$. Thus setting $c_n=3n/(3n-1), k=n$ in
   \eqref{pbound}, we see that \eqref{8} is valid for any $p$ satisfying
\begin{align}
\label{pboundgen}
   p \geq \max_{n \geq 2} \big \{ \big(1-B\big((1+\frac 1n)^{1/2} \big ) \big )^{-1} \big \}
\end{align}
  where
\begin{align*}
   B(x)= \frac 23+\frac 2{3} \cdot \frac {1}{(x^4-4x^2)/(1+4x)-1}.
\end{align*}
  Note that
\begin{align*}
    \left(\frac {x^4-4x^2}{1+4x} \right )'=
\frac {4x(x(3x^2-4)+x^2-2)}{(1+4x)^2}.
\end{align*}

    It follows that the right-hand side expression above is $<0$ when $x^2 \leq 4/3$, which implies that $B(x)$ is an increasing function of $x$ when $x
    \leq \sqrt{4/3}$. We then deduce from \eqref{pboundgen} that
\begin{align*}
   p \geq \max_{n=2,3} \big \{ \big(1-B\big((1+\frac 1n)^{1/2} \big ) \big )^{-1} \big \} \approx 1.3497.
\end{align*}

   The first assertion of Theorem \ref{mainthm} now follows from the above estimation and this completes the proof of Theorem \ref{mainthm}.


\vspace*{.5cm}
\noindent\begin{tabular}{p{8cm}p{8cm}}
School of Mathematical Sciences & Academy of Mathematics and Systems Science \\
Beihang University &  Chinese Academy of Sciences\\
Beijing 100191, China & Beijing 100190, China \\
Email: {\tt penggao@buaa.edu.cn} & Email: {\tt zhaohuayu17@mails.ucas.ac.cn} \\
\end{tabular}

\end{document}